\documentclass{amsart}
\usepackage{amsmath}
\usepackage{amsfonts}

\newtheorem{theorem}{Theorem}
\theoremstyle{plain}

\newtheorem{corollary}{Corollary}

\numberwithin{equation}{section}
\input{tcilatex}

\begin{document}
\title[Ostrowski-type inequality]{On the Ostrowski type integral inequality
for double integrals}
\author{Mehmet Zeki SARIKAYA}
\address{Department of Mathematics, \ Faculty of Science and Arts, D\"{u}zce
University, D\"{u}zce-TURKEY}
\email{sarikayamz@gmail.com, sarikaya@aku.edu.tr}

\begin{abstract}
In this note, we establish new an inequality of Ostrowski-type for double
integrals involving functions of two independent variables by using fairly
elementary analysis.
\end{abstract}

\maketitle

\section{Introduction}

In 1938, the classical integral inequality established by Ostrowski \cite%
{Ostrowski} as follows:

\begin{theorem}
\label{z1} Let $f:[a,b]\mathbb{\rightarrow R}$ be a differentiable mapping
on $(a,b)$ whose derivative $f^{^{\prime }}:(a,b)\mathbb{\rightarrow R}$ is
bounded on $(a,b),$ i.e., $\left\Vert f^{\prime }\right\Vert _{\infty }=%
\underset{t\in (a,b)}{\overset{}{\sup }}\left\vert f^{\prime }(t)\right\vert
<\infty .$ Then we have the inequality:%
\begin{equation}
\left\vert f(x)-\frac{1}{b-a}\int\limits_{a}^{b}f(t)dt\right\vert \leq \left[
\frac{1}{4}+\frac{(x-\frac{a+b}{2})^{2}}{(b-a)^{2}}\right] (b-a)\left\Vert
f^{\prime }\right\Vert _{\infty }  \label{1}
\end{equation}%
for all$\ x\in \lbrack a,b].$ The constant $\frac{1}{4}$ is the best
possible.
\end{theorem}

In a recent paper \cite{Barnett}, Barnett and Dragomir proved the following
Ostrowski type inequality for double integrals

\begin{theorem}
Let $f:[a,b]\times \lbrack c,d]\mathbb{\rightarrow R}$ be continuous on $%
[a,b]\times \lbrack c,d],$ $f_{x,y}^{\prime \prime }=\frac{\partial ^{2}f}{%
\partial x\partial y}$ exists on $(a,b)\times (c,d)$ and is bounded$,$ i.e., 
$\left\Vert f_{x,y}^{\prime \prime }\right\Vert _{\infty }=\underset{%
(x,y)\in (a,b)\times (c,d)}{\overset{}{\sup }}\left\vert \frac{\partial
^{2}f(x,y)}{\partial x\partial y}\right\vert <\infty .$ Then, we have the
inequality:%
\begin{eqnarray}
&&\left\vert
\int\limits_{a}^{b}\int\limits_{c}^{d}f(s,t)dtds-(d-c)(b-a)f(x,y)\right. 
\notag \\
&&  \notag \\
&&\ \ \ \ \ \ \ \ \ \ \ \ \ \ \ \ \ \ \ \ \ \ \ \ \ \ \ \left. -\left[
(b-a)\int\limits_{c}^{d}f(x,t)dt+(d-c)\int\limits_{a}^{b}f(s,y)ds\right]
\right\vert  \label{2} \\
&&  \notag \\
&\leq &\left[ \frac{1}{4}(b-a)^{2}+(x-\frac{a+b}{2})^{2}\right] \left[ \frac{%
1}{4}(d-c)^{2}+(y-\frac{d+c}{2})^{2}\right] \left\Vert f_{x,y}^{\prime
\prime }\right\Vert _{\infty }  \notag
\end{eqnarray}%
for all $(x,y)\in \lbrack a,b]\times \lbrack c,d].$
\end{theorem}

In \cite{Barnett}, the inequality (\ref{2}) is established by the use of
integral identity involving Peano kernels. In \cite{Pachpatte1}, Pachpatte
obtained an inequality in the view (\ref{2}) by using elementary analysis.
The interested reader is also refered to (\cite{Barnett}, \cite{Dragomir}, 
\cite{Pachpatte}-\cite{Ujevic}) for Ostrowski type inequalities \ in several
independent variables.

The main aim of this note is to establish a new Ostrowski type inequality
for double integrals involving functions of two independent variables and
their partial derivatives.

\section{Main Result}

\begin{theorem}
\label{z} Let $f:[a,b]\times \lbrack c,d]\mathbb{\rightarrow R}$ be an
absolutely continuous fuction such that the partial derivative of order $2$
exist and is bounded, i.e.,%
\begin{equation*}
\left\Vert \frac{\partial ^{2}f(t,s)}{\partial t\partial s}\right\Vert
_{\infty }=\underset{(x,y)\in (a,b)\times (c,d)}{\overset{}{\sup }}%
\left\vert \frac{\partial ^{2}f(t,s)}{\partial t\partial s}\right\vert
<\infty
\end{equation*}%
for all $(t,s)\in \lbrack a,b]\times \lbrack c,d].$ Then, we have%
\begin{equation}
\begin{array}{l}
\left\vert (\beta _{1}-\alpha _{1})(\beta _{2}-\alpha _{2})f(\frac{a+b}{2},%
\frac{c+d}{2})+H(\alpha _{1},\alpha _{2},\beta _{1},\beta _{2})+G(\alpha
_{1},\alpha _{2},\beta _{1},\beta _{2})\right. \\ 
\\ 
-(\beta _{2}-\alpha _{2})\dint\limits_{a}^{b}f(t,\frac{c+d}{2})dt-(\beta
_{1}-\alpha _{1})\dint\limits_{c}^{d}f(\frac{a+b}{2},s)ds \\ 
\\ 
-\dint\limits_{a}^{b}[(\alpha _{2}-c)f(t,c)+(d-\beta
_{2})f(t,d)]dt-\dint\limits_{c}^{d}[(\alpha _{1}-a)f(a,s)+(b-\beta
_{1})f(b,s)]ds \\ 
\\ 
\left. +\dint\limits_{a}^{b}\dint\limits_{c}^{d}f(t,s)dsdt\right\vert \leq 
\left[ \dfrac{(\alpha _{1}-a)^{2}+(b-\beta _{1})^{2}}{2}+\dfrac{(a+b-2\alpha
_{1})^{2}+(a+b-2\beta _{1})^{2}}{8}\right] \\ 
\\ 
\times \left[ \dfrac{(\alpha _{2}-c)^{2}+(d-\beta _{2})^{2}}{2}+\dfrac{%
(c+d-2\alpha _{2})^{2}+(c+d-2\beta _{2})^{2}}{8}\right] \left\Vert \dfrac{%
\partial ^{2}f(t,s)}{\partial t\partial s}\right\Vert _{\infty }%
\end{array}
\label{4}
\end{equation}%
for all $(\alpha _{1},\alpha _{2}),(\beta _{1},\beta _{2})\in \lbrack
a,b]\times \lbrack c,d]$ with $\alpha _{1}<\beta _{1},\ \alpha _{2}<\beta
_{2}\ $where%
\begin{equation}
\begin{array}{l}
H(\alpha _{1},\alpha _{2},\beta _{1},\beta _{2})=(\alpha _{1}-a)[(\alpha
_{2}-c)f(a,c)+(d-\beta _{2})f(a,d)] \\ 
\\ 
+(b-\beta _{1})[(\alpha _{2}-c)f(b,c)+(d-\beta _{2})f(b,d)]%
\end{array}
\label{a}
\end{equation}%
and%
\begin{equation}
\begin{array}{l}
G(\alpha _{1},\alpha _{2},\beta _{1},\beta _{2})=(\beta _{1}-\alpha _{1}) 
\left[ (\alpha _{2}-c)f(\frac{a+b}{2},c)+(d-\beta _{2})f(\frac{a+b}{2},d)%
\right] \\ 
\\ 
+(\beta _{2}-\alpha _{2})\left[ (\alpha _{1}-a)f(a,\frac{c+d}{2})+(b-\beta
_{1})f(b,\frac{c+d}{2})\right] .%
\end{array}
\label{b}
\end{equation}
\end{theorem}

\begin{proof}
We define the following functions:%
\begin{equation*}
\begin{array}{ccc}
p(a,b,\alpha _{1},\beta _{1},t) & = & \left\{ 
\begin{array}{lll}
t-\alpha _{1}, & t\in \lbrack a,\frac{a+b}{2}] &  \\ 
&  &  \\ 
t-\beta _{1}, & t\in (\frac{a+b}{2},b] & 
\end{array}%
\right.%
\end{array}%
\end{equation*}%
and%
\begin{equation*}
\begin{array}{ccc}
q(c,d,\alpha _{2},\beta _{2},s) & = & \left\{ 
\begin{array}{lll}
s-\alpha _{2}, & s\in \lbrack c,\frac{c+d}{2}] &  \\ 
&  &  \\ 
s-\beta _{2}, & s\in (\frac{c+d}{2},d] & 
\end{array}%
\right.%
\end{array}%
\end{equation*}%
for all $(\alpha _{1},\alpha _{2}),(\beta _{1},\beta _{2})\in \lbrack
a,b]\times \lbrack c,d]\ $with $\alpha _{1}<\beta _{1},\ \alpha _{2}<\beta
_{2}.$ Thus, by definitions of $p(a,b,\alpha _{1},\beta _{1},t)$ and $%
q(c,d,\alpha _{2},\beta _{2},s),$ we have%
\begin{equation}
\begin{array}{l}
\dint\limits_{a}^{b}\dint\limits_{c}^{d}p(a,b,\alpha _{1},\beta
_{1},t)q(c,d,\alpha _{2},\beta _{2},s)\dfrac{\partial ^{2}f(t,s)}{\partial
t\partial s}dsdt=\dint\limits_{a}^{\frac{a+b}{2}}\dint\limits_{c}^{\frac{c+d%
}{2}}(t-\alpha _{1})(s-\alpha _{2})\dfrac{\partial ^{2}f(t,s)}{\partial
t\partial s}dsdt \\ 
\\ 
+\dint\limits_{a}^{\frac{a+b}{2}}\dint\limits_{\frac{c+d}{2}}^{d}(t-\alpha
_{1})(s-\beta _{2})\dfrac{\partial ^{2}f(t,s)}{\partial t\partial s}%
dsdt+\dint\limits_{\frac{a+b}{2}}^{b}\dint\limits_{c}^{\frac{c+d}{2}%
}(t-\beta _{1})(s-\alpha _{2})\dfrac{\partial ^{2}f(t,s)}{\partial t\partial
s}dsdt \\ 
\\ 
+\dint\limits_{\frac{a+b}{2}}^{b}\dint\limits_{\frac{c+d}{2}}^{d}(t-\beta
_{1})(s-\beta _{2})\dfrac{\partial ^{2}f(t,s)}{\partial t\partial s}dsdt.%
\end{array}
\label{5}
\end{equation}%
Integrating by parts, we can state:%
\begin{equation}
\begin{array}{l}
\dint\limits_{a}^{\frac{a+b}{2}}\dint\limits_{c}^{\frac{c+d}{2}}(t-\alpha
_{1})(s-\alpha _{2})\dfrac{\partial ^{2}f(t,s)}{\partial t\partial s}dsdt=%
\dfrac{(a+b-2\alpha _{1})(c+d-2\alpha _{2})}{4}f(\frac{a+b}{2},\frac{c+d}{2}%
)+\dint\limits_{a}^{\frac{a+b}{2}}\dint\limits_{c}^{\frac{c+d}{2}}f(t,s)dsdt
\\ 
\\ 
-\dfrac{(a-\alpha _{1})(c+d-2\alpha _{2})}{2}f(a,\frac{c+d}{2})-\dfrac{%
(a+b-2\alpha _{1})(c-\alpha _{2})}{2}f(\frac{a+b}{2},c)+(a-\alpha
_{1})(c-\alpha _{2})f(a,c) \\ 
\\ 
-\dint\limits_{a}^{\frac{a+b}{2}}[\dfrac{(c+d-2\alpha _{2})}{2}f(t,\frac{c+d%
}{2})-(c-\alpha _{2})f(t,c)]dt-\dint\limits_{c}^{\frac{c+d}{2}}[\dfrac{%
(a+b-2\alpha _{1})}{2}f(\frac{a+b}{2},s)-(a-\alpha _{1})f(a,s)]ds.%
\end{array}
\label{6}
\end{equation}%
\begin{equation}
\begin{array}{l}
\dint\limits_{a}^{\frac{a+b}{2}}\dint\limits_{\frac{c+d}{2}}^{d}(t-\alpha
_{1})(s-\beta _{2})\dfrac{\partial ^{2}f(t,s)}{\partial t\partial s}dsdt=-%
\dfrac{(a+b-2\alpha _{1})(c+d-2\beta _{2})}{4}f(\frac{a+b}{2},\frac{c+d}{2}%
)+\dint\limits_{a}^{\frac{a+b}{2}}\dint\limits_{\frac{c+d}{2}}^{d}f(t,s)dsdt
\\ 
\\ 
+\dfrac{(a-\alpha _{1})(c+d-2\beta _{2})}{2}f(a,\frac{c+d}{2})+\dfrac{%
(a+b-2\alpha _{1})(d-\beta _{2})}{2}f(\frac{a+b}{2},d)+(a-\alpha
_{1})(d-\beta _{2})f(a,d) \\ 
\\ 
+\dint\limits_{a}^{\frac{a+b}{2}}[\dfrac{(c+d-2\beta _{2})}{2}f(t,\frac{c+d}{%
2})-(d-\beta _{2})f(t,d)]dt-\dint\limits_{\frac{c+d}{2}}^{d}[\dfrac{%
(a+b-2\alpha _{1})}{2}f(\frac{a+b}{2},s)-(a-\alpha _{1})f(a,s)]ds.%
\end{array}
\label{7}
\end{equation}%
\begin{equation}
\begin{array}{l}
\dint\limits_{\frac{a+b}{2}}^{b}\dint\limits_{c}^{\frac{c+d}{2}}(t-\beta
_{1})(s-\alpha _{2})\dfrac{\partial ^{2}f(t,s)}{\partial t\partial s}dsdt=-%
\dfrac{(a+b-2\beta _{1})(c+d-2\alpha _{2})}{4}f(\frac{a+b}{2},\frac{c+d}{2}%
)+\dint\limits_{\frac{a+b}{2}}^{b}\dint\limits_{c}^{\frac{c+d}{2}}f(t,s)dsdt
\\ 
\\ 
+\dfrac{(b-\beta _{1})(c+d-2\alpha _{2})}{2}f(b,\frac{c+d}{2})+\dfrac{%
(a+b-2\beta _{1})(c-\alpha _{2})}{2}f(\frac{a+b}{2},c)-(b-\beta
_{1})(c-\alpha _{2})f(b,c) \\ 
\\ 
-\dint\limits_{\frac{a+b}{2}}^{b}[\dfrac{(c+d-2\alpha _{2})}{2}f(t,\frac{c+d%
}{2})-(c-\alpha _{2})f(t,c)]dt+\dint\limits_{c}^{\frac{c+d}{2}}[\dfrac{%
(a+b-2\beta _{1})}{2}f(\frac{a+b}{2},s)-(b-\beta _{1})f(b,s)]ds.%
\end{array}
\label{8}
\end{equation}%
\begin{equation}
\begin{array}{l}
\dint\limits_{\frac{a+b}{2}}^{b}\dint\limits_{\frac{c+d}{2}}^{d}(t-\beta
_{1})(s-\beta _{2})\dfrac{\partial ^{2}f(t,s)}{\partial t\partial s}dsdt=%
\dfrac{(a+b-2\beta _{1})(c+d-2\beta _{2})}{4}f(\frac{a+b}{2},\frac{c+d}{2}%
)+\dint\limits_{\frac{a+b}{2}}^{b}\dint\limits_{\frac{c+d}{2}}^{d}f(t,s)dsdt
\\ 
\\ 
-\dfrac{(b-\beta _{1})(c+d-2\beta _{2})}{2}f(b,\frac{c+d}{2})-\dfrac{%
(a+b-2\beta _{1})(d-\beta _{2})}{2}f(\frac{a+b}{2},d)+(b-\beta _{1})(d-\beta
_{2})f(b,d) \\ 
\\ 
+\dint\limits_{\frac{a+b}{2}}^{b}[\dfrac{(c+d-2\beta _{2})}{2}f(t,\frac{c+d}{%
2})-(d-\beta _{2})f(t,d)]dt+\dint\limits_{\frac{c+d}{2}}^{d}[\dfrac{%
(a+b-2\beta _{1})}{2}f(\frac{a+b}{2},s)-(b-\beta _{1})f(b,s)]ds.%
\end{array}
\label{9}
\end{equation}%
Adding (\ref{6})-(\ref{9}) and rewriting, we easily deduce:%
\begin{equation}
\begin{array}{l}
\dint\limits_{a}^{b}\dint\limits_{c}^{d}p(a,b,\alpha _{1},\beta
_{1},t)q(c,d,\alpha _{2},\beta _{2},s)\dfrac{\partial ^{2}f(t,s)}{\partial
t\partial s}dsdt=(\beta _{1}-\alpha _{1})(\beta _{2}-\alpha _{2})f(\frac{a+b%
}{2},\frac{c+d}{2})+H(\alpha _{1},\alpha _{2},\beta _{1},\beta _{2}) \\ 
\\ 
+G(\alpha _{1},\alpha _{2},\beta _{1},\beta _{2})-(\beta _{2}-\alpha
_{2})\dint\limits_{a}^{b}f(t,\frac{c+d}{2})dt-(\beta _{1}-\alpha
_{1})\dint\limits_{c}^{d}f(\frac{a+b}{2},s)ds \\ 
\\ 
-\dint\limits_{a}^{b}[(\alpha _{2}-c)f(t,c)+(d-\beta
_{2})f(t,d)]dt-\dint\limits_{c}^{d}[(\alpha _{1}-a)f(a,s)+(b-\beta
_{1})f(b,s)]ds \\ 
\\ 
+\dint\limits_{a}^{b}\dint\limits_{c}^{d}f(t,s)dsdt%
\end{array}
\label{10}
\end{equation}%
where $H(\alpha _{1},\alpha _{2},\beta _{1},\beta _{2})$ and $G(\alpha
_{1},\alpha _{2},\beta _{1},\beta _{2})$ defined by (\ref{a}) and (\ref{b}),
respectively. Now, using the identitiy (\ref{10}), it follows that%
\begin{equation}
\begin{array}{l}
\left\vert (\beta _{1}-\alpha _{1})(\beta _{2}-\alpha _{2})f(\frac{a+b}{2},%
\frac{c+d}{2})+H(\alpha _{1},\alpha _{2},\beta _{1},\beta
_{2})+\dint\limits_{a}^{b}\dint\limits_{c}^{d}f(t,s)dsdt\right. \\ 
\\ 
+G(\alpha _{1},\alpha _{2},\beta _{1},\beta _{2})-(\beta _{2}-\alpha
_{2})\dint\limits_{a}^{b}f(t,\frac{a+b}{2})dt-(\beta _{1}-\alpha
_{1})\dint\limits_{c}^{d}f(x,\frac{c+d}{2})ds \\ 
\\ 
\left. -\dint\limits_{a}^{b}[(\alpha _{2}-c)f(t,c)+(d-\beta
_{2})f(t,d)]dt-\dint\limits_{c}^{d}[(\alpha _{1}-a)f(a,s)+(b-\beta
_{1})f(b,s)]ds\right\vert \\ 
\\ 
\leq \dint\limits_{a}^{b}\dint\limits_{c}^{d}\left\vert p(a,b,\alpha
_{1},\beta _{1},t)\right\vert \left\vert q(c,d,\alpha _{2},\beta
_{2},s)\right\vert \left\vert \dfrac{\partial ^{2}f(t,s)}{\partial t\partial
s}\right\vert dsdt \\ 
\\ 
\leq \left\Vert \dfrac{\partial ^{2}f(t,s)}{\partial t\partial s}\right\Vert
_{\infty }\dint\limits_{a}^{b}\dint\limits_{c}^{d}\left\vert p(a,b,\alpha
_{1},\beta _{1},t)\right\vert \left\vert q(c,d,\alpha _{2},\beta
_{2},s)\right\vert dsdt.%
\end{array}
\label{12}
\end{equation}%
On the other hand, we get%
\begin{eqnarray}
\dint\limits_{a}^{b}\left\vert p(a,b,\alpha _{1},\beta _{1},t)\right\vert dt
&=&\int\limits_{a}^{\frac{a+b}{2}}\left\vert t-\alpha _{1}\right\vert
dt+\dint\limits_{\frac{a+b}{2}}^{b}\left\vert t-\beta _{1}\right\vert dt 
\notag \\
&=&\int\limits_{a}^{\alpha _{1}}\left( \alpha _{1}-t\right)
dt+\int\limits_{\alpha _{1}}^{\frac{a+b}{2}}\left( t-\alpha _{1}\right)
dt+\dint\limits_{\frac{a+b}{2}}^{\beta _{1}}\left( \beta _{1}-t\right)
dt+\dint\limits_{\beta _{1}}^{b}\left( t-\beta _{1}\right) dt  \notag \\
&=&\frac{(\alpha _{1}-a)^{2}+(b-\beta _{1})^{2}}{2}+\frac{(a+b-2\alpha
_{1})^{2}+(a+b-2\beta _{1})^{2}}{8}  \label{13}
\end{eqnarray}%
and similarly,%
\begin{eqnarray}
\dint\limits_{c}^{d}\left\vert q(a,b,\alpha _{1},\beta _{1},t)\right\vert dt
&=&\int\limits_{c}^{\frac{c+d}{2}}\left\vert s-\alpha _{2}\right\vert
ds+\dint\limits_{\frac{c+d}{2}}^{d}\left\vert s-\beta _{2}\right\vert ds 
\notag \\
&=&\frac{(\alpha _{2}-c)^{2}+(d-\beta _{2})^{2}}{2}+\frac{(c+d-2\alpha
_{2})^{2}+(c+d-2\beta _{2})^{2}}{8}.  \label{14}
\end{eqnarray}%
Using (\ref{13}) and (\ref{14}) in (\ref{12}), we see that (\ref{4}) holds.
\end{proof}

\begin{corollary}
Under the assumptions of Theorem \ref{z}, we have%
\begin{equation}
\begin{array}{l}
\left\vert (b-a)(d-c)f(\dfrac{a+b}{2},\dfrac{c+d}{2})-(d-c)\dint%
\limits_{a}^{b}f(t,\dfrac{c+d}{2})dt-(b-a)\dint\limits_{c}^{d}f(\dfrac{a+b}{2%
},s)ds\right. \\ 
\\ 
\ \ \ \ \ \ \left.
+\dint\limits_{a}^{b}\dint\limits_{c}^{d}f(t,s)dsdt\right\vert \leq \dfrac{1%
}{16}\left\Vert \dfrac{\partial ^{2}f(t,s)}{\partial t\partial s}\right\Vert
_{\infty }(b-a)^{2}(d-c)^{2}.%
\end{array}
\label{15}
\end{equation}
\end{corollary}

\begin{proof}
We choose $\alpha _{1}=a,\ \beta _{1}=b,\ \alpha _{2}=c$ and $\beta _{2}=d$
in (\ref{4}), then we see that (\ref{15}) holds.
\end{proof}

\begin{corollary}
Under the assumptions of Theorem \ref{z}, we have%
\begin{equation}
\begin{array}{l}
\left\vert \dfrac{(b-a)(d-c)}{4}\left[ f(a,c)+f(a,d)+f(b,c)+f(b,d)\right] -%
\dfrac{(d-c)}{2}\dint\limits_{a}^{b}\left[ f(t,c)+f(t,d)\right] dt\right. \\ 
\\ 
\ \ \ \ \ \ \left. -\dfrac{(b-a)}{2}\dint\limits_{c}^{d}\left[ f(a,s)+f(b,s)%
\right] ds+\dint\limits_{a}^{b}\dint\limits_{c}^{d}f(t,s)dsdt\right\vert \\ 
\\ 
\ \ \ \ \ \ \leq \dfrac{1}{16}\left\Vert \dfrac{\partial ^{2}f(t,s)}{%
\partial t\partial s}\right\Vert _{\infty }(b-a)^{2}(d-c)^{2}.%
\end{array}
\label{16}
\end{equation}
\end{corollary}

\begin{proof}
We choose $\alpha _{1}=\beta _{1}=\frac{a+b}{2},\ \alpha _{2}=\beta _{2}=%
\frac{c+d}{2}$ in (\ref{4}), then we see that (\ref{16}) holds.
\end{proof}

\end{document}